\documentclass{amsart}
\usepackage{amsmath, amssymb, mathrsfs, amsfonts,  amsthm} 
\usepackage [all]{xy}
\newcommand{\ls}{\Omega \textrm {Ham} (M, \omega)} 

\newcommand{\delbar}{\bar{\partial}}

\DeclareMathOperator{\area}{area}

\DeclareMathOperator{\Diff}{Diff}
\DeclareMathOperator{\energy}{c-energy}
\newtheorem {theorem} {Theorem} [section] 
\newtheorem {conjecture} [theorem] {Conjecture}
\newtheorem{lemma}[theorem] {Lemma} 

\newtheorem {question}  [theorem] {Question}
\newtheorem {definition} [theorem] {Definition} 
\newtheorem {proposition}  [theorem]{Proposition} 
\newtheorem {corollary}[theorem]  {Corollary}

\newtheorem {remark} [theorem] {Remark}
\numberwithin {equation} {section}

\DeclareMathOperator{\grad}{grad}
\begin{document}
\author {Yasha Savelyev}
\title   {Virtual Morse theory on $\ls$}
\today       
\begin{abstract} We relate previously defined quantum characteristic
classes to Morse theoretic aspects of the Hofer length functional on $\ls$. As
an application we prove a theorem which can be interpreted as stating that this
functional behaves ``virtually'' as a perfect Morse-Bott functional with a flow.
This can be applied to study  topology and Hofer geometry of $ \text {Ham}(M, \omega)$. We
also use this to give a prediction for the index of some geodesics for this functional, which was
recently partially verified by Yael Karshon and Jennifer Slimowitz. 
 \end{abstract}
\maketitle

\section {Introduction}  
Hamiltonian fibrations  over Riemann
surfaces form a rich object from the point of view of Gromov Witten
theory, as the properties of holomorphic sections of such fibrations can be 
closely tied with the underlying geometry of the fibration. Here by
geometry we may mean a number of
things like curvature properties of Hamiltonian connections, geometry of
the coupling forms in the sense of \cite{GuileminSternberg}, as well as the
associated  quantities like Gromov's $K$-area, (cf. \cite{polterovich}) and the
related notion of area of a Hamiltonian fibration. 

Gromov Witten theory of Hamiltonian fibrations $M \hookrightarrow X
\xrightarrow{\pi} \Sigma$ fits into a certain $2d$ Hamiltonian cohomological
field theory, or even more generally into a
string background \cite{fieldtheory}, considerably extending Gromov Witten
theory of $(M, \omega)$. In this paper we will be concerned with a fairly small, but 
geometrically still rich part of this theory by restricting to genus 0, one
input one  output part of the data of field theory. With some further restrictions,  the data one gets is
a ring homomorphism defined in \cite{GS}: $$\Psi: H_*
(\Omega \text {Ham}(M, \omega), \mathbb{Q}) \to QH  (M).$$  

Geometry of Hamiltonian fibrations over $S ^{2}$ can be tied with Hofer
geometry of loops in $ \text {Ham}(M, \omega)$, and we are going to relate
$\Psi$ to a kind of virtual Morse theory of the positive Hofer
length functional, $L ^{+}: \ls \to \mathbb{R}$, (see \eqref{eq.pos.length}).

\subsection {Morse theory on $\ls$ and $\Psi$}
Let $h: B \to \Omega \text {Ham}(M, \omega)$ be a smooth cycle, where $B$ is a
closed oriented smooth manifold. Let 
\begin{equation} \label {eq.definitionPh} P _{h}=B \times M \times D ^{2} _{0} 
\bigcup B \times M \times   D ^{2} _{\infty}/
\sim,
\end{equation}
where $ (b, x, 1, \theta)_0 \sim (b, h _{b, \theta}(x), 1, \theta)
_{\infty}$,
using the polar coordinates $(r, 2 \pi
 \theta)$. We get a bundle $$ {p}: P
_{h} \to B,$$ with fiber 
modelled by a Hamiltonian fibration $
M \hookrightarrow X \xrightarrow{\pi} S^2$. 
A Hamiltonian connection or equivalently, a coupling form
(see \cite[Theorem 6.21]{MS2}) on the Hamiltonian bundle 
\begin{equation} \label {eq.ham.bundle} M \hookrightarrow P _{h} \to B \times S
^{2}
\end{equation} induces a family of complex structures $\{X _{b}=p ^{-1} (b), J
_{b}\}$.
The map $\Psi$ is defined by
counting fiber-wise (i.e. vertical) holomorphic curves in $p: P _{h} \to B$,
with some constraints. The details are given in Section \ref{section.prelims}. 

Let $\gamma: S ^{1} \to \text {Ham}(M, \omega)$  be a one parameter subgroup,
which is always assumed in this paper to be generated by a Morse Hamiltonian
$H$. The loop $\gamma$ is a smooth point of the functional $L ^{+}$ and is
critical, see Ustilovsky \cite{U}. Our focus will be on smooth cycles $h: B \to
\ls$ that resemble unstable manifolds for the functional $L ^{+}$ of $\gamma$,  
in the sense below. 
\begin{definition}  We say that a smooth map $h: B \to
\ls$ is Morse at $\gamma$, if the pullback of $L ^{+}$ to $B$ attains its
maximum at a unique point $\max \in B$ such that $L ^{+}$ is Morse at $\max$ and
$h (\max)=\gamma$. 
\end{definition}
To make use of this structure we construct a coupling form on $P _{h}$
naturally adapted to the above properties of $h$. As a consequence, for the
induced family $ \{J ^{h} _{b}\}$ of complex structures on $P _{h}$ all
vertical holomorphic curves in $P _{h}$ of ``maximum allowed $\energy$''
(Definition \ref{definition.energy}) localize over $\max$ and in fact
correspond to a single, distinguished flat section $\sigma _{\max}$ of the fiber
$X _{\max}$. 

One would then like that this curve is
persistent and contributes to the invariant $\Psi$. However, without further
restriction on $(h, B)$ this is not true as one can see from simple heuristic 
intuition: the map $h: B \to \ls$ which by assumption is Morse at
$\gamma$ can be homotoped below the energy level $L ^{+} (\gamma)$, unless its dimension is that of the unstable manifold
of $\gamma$ for the functional $L ^{+}$, (provided one can make sense of such
an unstable manifold). Once this
happens, the part of the invariant $\Psi$ corresponding to curves of the same class as $\sigma _{\max}$ has to vanish. (This is explained in the proof of Theorem
\ref{theorem.index}). There is a more formal obstruction: dimension of $B$ has
to be the dimension of the cokernel of the linearized Cauchy-Riemann operator
corresponding to the pair $(\sigma _{\max}, X _{\max})$, for
otherwise the index of the overall problem is not zero and our moduli space does
not even have the expected dimension. 

The index of the above Cauchy-Riemann operator will be
denoted by $I ^{virt} (\gamma)$ and we call this the \emph{virtual index of
$\gamma$}. Indeed, the name arises from intuition that the above necessary
conditions are the same and this is verified by Theorem
\ref{theorem.index} below. 
We show in \cite [Section 5.1]{GS}:
\begin{equation} \label {eq.virtualindex} I ^{virt} (\gamma)= 
\sum _{\substack {1 \leq i \leq n \\ k_i  \leq -1}} 2(|k_i|  -1),
\end{equation}   
where $k_i$ are the \emph{weights} of the linearized action of $\gamma$ on
$T_{x_{\max}} M$, the tangent space to the maximum: $x_{\max}$, of the
generating function $H$ of $\gamma$. This is a single point  since $H$ is Morse
and the level sets of $H$ must be connected by the Atiyah-Guillemin-Sternberg
convexity theorem. To define these weights one takes an $S^1$ equivariant
orientation preserving identification of $T _{x_{\max}}$ with $ \mathbb{C} ^{n}$, which
 splits into $\gamma$ invariant 1 complex dimensional subspaces $N _{k_i}$, on which
$\gamma$ is acting by \begin{equation} \label {eq.weights} v \mapsto e ^{2\pi ik_i \theta}v.
\end{equation}
These $k_i$ are then defined to be the weights of the circle action $\gamma$. 
Our conventions are
\begin{align} \omega (X _{H}, \cdot) = -dH (\cdot)\\  
\omega (\cdot, J \cdot)>0.
\end{align} 
With these conventions the above weights are negative.  
Let $L ^{+} (\gamma)$ denote the positive Hofer length of $\gamma$, see Section
\ref{sec.hofer}. The following is our main main technical result proved in 
in Section \ref{section.hoferandpsi}:
 \begin{theorem}
\label{theorem.main} Let $h: B _{\gamma} \to \ls$ be Morse at $\gamma$ and
such that $I ^{virt} (\gamma)= \dim B _{\gamma}$, then  \begin {equation} 0 \neq
\Psi (h) =[\pm pt] \cdot e ^{iL ^{+} (\gamma)} + \text {corrections} \quad \in QH
(M).
\end {equation} 
\end{theorem}
In \cite{GS} we also studied what we called the max length measure of $h:B \to
\ls$ which is defined by 
\begin{equation} \label {eq.max.length} L^+(h) \equiv \max _{b \in B} L ^{+}
( h (b)).
\end{equation}
 \begin{corollary} \label{corollary.max.length} Let $h: B _{\gamma} \to \ls$
be as in Theorem \ref{theorem.main}, then the cycle $h: B _{\gamma} \to \Omega \text
{Ham}(M, \omega)$ does not vanish in rational homology and moreover minimizes
the max length measure in its homology class.
\end {corollary}
When $\gamma$ is generated by a Morse Hamiltonian $H$, $\gamma$ is a smooth point of $L ^{+}$ by
Ustilovsky's work \cite{U} and we can make the following definition.
 \begin{definition} \label{definition.hofer.index} Let
$\gamma$ be a one parameter subgroup of $\text {Ham}(M, \omega)$, where $(M, \omega)$ is any closed symplectic manifold. We define the
Hofer index $I ^{H} (\gamma)$ to be the maximal dimensional subspace of
$T_{\gamma} \ls$ on which the Hessian of $L ^{+}$ is negative definite. 
\end{definition}
Of course the above index can apriori be infinite. On the other hand we have:
\begin{theorem}  \label{theorem.index}  Let $h: B _{\gamma} \to \ls$ be as in
Theorem \ref{theorem.main} then $$I ^{H} (\gamma)= I ^{virt} (\gamma)=\dim B _{\gamma}.$$
\begin{conjecture} \label{conjecture.index} Let $\gamma$ be a Hamiltonian circle
action on $(M, \omega)$ generated by a Morse Hamiltonian, then 
\begin {equation} I ^{H} (\gamma)= I ^{virt} (\gamma).
\end {equation}  
\end{conjecture}
\end{theorem}
Yael Karshon and Jennifer Slimowitz~\cite{Yael} have recently verified that 
$I ^{H} (\gamma) \geq I ^{virt} (\gamma)$. They explicitly construct a local
family deforming $\gamma$ of dimension $I ^{virt} (\gamma)$,  so that the
Hessian of $L ^{+}$ on the tangent space to this family is negative definite.
Following a suggestion of Leonid Polterovich I know expect to be able to prove
this conjecture, using classical calculus of variations and Duistermaat's
 theorem relating Morse index and Maslov index, \cite{Duistermaat}.
\subsection {Hofer functional and ``virtual Morse theory''}
Given the conjecture above, Corollary \ref{corollary.max.length} can be
restated with Morse index of $\gamma$ replacing virtual index of $\gamma$. 
However, this now raises an  observation. Since the Morse indexes
are even,  the first part of the statement of the corollary would hold
automatically if Hofer length functional had well behaved negative gradient flow
and an associated Morse-Bott complex computing homology of $\ls$. 
Nothing like this could remotely be true directly as the Hofer length
functional is extremely degenerate and poorly behaved analytically.
 Yet from
the point of view of Corollary \ref{corollary.max.length} something like this
happens virtually.

\subsection {Generalizations to path spaces}
All of the results outlined in this section have appropriate generalizations to
path spaces. For example Theorem \ref{theorem.main} can be stated for cycles in
path space $\Omega _{\phi_1, \phi_2}$ Morse at $\gamma$, where $\gamma$
is an autonomous geodesic (generated by Morse $H$) between
non-conjugate $\phi_1, \phi_2 \in \text {Ham}(M, \omega)$. Here non-conjugate is in Hofer geometry sense, which
amounts to the condition that the linearized flow at the maximizer of $H$ has no
periodic orbits with period less than $1$.  In this case one must use Floer homology instead of quantum
homology, and the role of fibrations $M \hookrightarrow X _{b} \to S
^{2}$ is played by fibrations $$M \hookrightarrow X _{b} \to \mathbb{R} \times S
^{1}$$ with asymptotic ($r \mapsto \pm \infty$, with $r, \theta$ coordinates on
$\mathbb{R} \times S ^{1}$) boundary monodromy maps $\phi_1, \phi_2$, but
otherwise the statements and proofs are completely analogous. It is even likely that one can extend from autonomous geodesics to any geodesics
of the Hofer length functional, (which are generated by quasi-autonomous time
dependent Hamiltonians, see \cite{U} and \cite{ML1}). However in this case the
proofs may need to slightly change, since we give up some symmetry.
\subsection {A special case: $\Omega \text {Ham}(G/T)$}
\label{section.special.case} Consider the Hamiltonian action of $G$ on $G/T$. 
In Section \ref{section.energy.flow} we relate the ``Morse theory'' for the
functional $L ^{+}: \Omega G \to \mathbb{R}$ pulled back from
$\Omega \text {Ham}(G/T)$ and the energy functional $E: \Omega G \to \mathbb{R}$
induced by a bi-invariant metric on $G$. The latter functional is amazingly
well behaved. It is Morse-Bott, the ``smooth'' negative gradient flow (i.e.
energy flow) exists for all time and the unstable manifolds of critical level
sets are complex submanifolds. (These appear to be rather deep facts of life, see \cite{Press} and or
\cite{Segal}.)
    
Let $f: \Omega G \to \Omega \text {Ham}(G/T)$ be the map induced by the
Hamiltonian action. The first theorem follows
from Corollaries \ref{corollary.ultimate}, \ref{corollary.max.length}.
\begin{theorem} \label {theorem.application} Let $G$ be a semi simple Lie group,
$\gamma$ an $S^1$ subgroup of $G$ whose centralizer is the torus, and $h:B
_{\gamma} \to G$ the pseudocycle corresponding to the unstable manifold of $\gamma$ in $\Omega G$
for the  Riemannian energy functional.  Then the pseudocycle $f \circ h: B
_{\gamma} \to \Omega \text {Ham}(G/T)$ is non-vanishing in $H _{\dim B
_{\gamma}} (\Omega \text {Ham}(G/T), \mathbb{Q})$ and moreover it minimizes the max-length measure in its
homology class, (see eq. \eqref{eq.max.length}).
\end{theorem}

Theorem \ref{theorem.index} together with Corollary \ref{corollary.ultimate}
gives: \begin{theorem} Let $\gamma$ be as in the above theorem, then $I ^{H} (f
\circ \gamma)$ is the Riemannian index of $\gamma$, i.e. the index of the
geodesic $\gamma$ as a critical point of the Riemannian energy functional on
$\Omega G$.
\end{theorem} 
 Alexander Givental asked me the following natural question:  
 \begin{question} Does the first part of Theorem \ref{theorem.application}
 remain true if $\Omega \text {Ham}(G/T)$ is replaced with $\Omega {\Diff}(G/T)$?
 \end{question}
 My feeling is that the answer is no, however not much is known about
topology of diffeomorphism groups of higher dimensional manifolds.
\subsection*
{Acknowledgements} This is part of the author's doctoral research at Stony Brook university. I would like
to thank my advisor Dusa McDuff and Aleksey Zinger for numerous
suggestions and discussions. Aleksey Zinger in particular for noticing a
serious issue in an earlier version of the manuscript. As well as Michael Entov
and Yael Karshon for helping me out with some properties of normalized Hamiltonian functions.
\section {Preliminaries and the map $\Psi$} \label{section.prelims}
\subsection {The group of Hamiltonian symplectomorphisms and Hofer metric}
\label{sec.hofer} Given a smooth function $H _{t}: M
\to \mathbb{R}$, $0 \leq t\leq 1$,  there is  an associated time dependent
Hamiltonian vector field $X _{t}$, $0 \leq t \leq 1$,  defined by 
\begin{equation}  \label {equation.ham.flow} \omega (X _{t}, \cdot)=-dH _{t} (
\cdot).
\end{equation}
The vector field $X _{t}$ generates a path $\gamma_{t}$, $0 \leq t \leq1$, in
$\text
{Diff}(M, \omega)$. Given such a path $\gamma_t$, its end point $\gamma_1$ is
called a
Hamiltonian symplectomorphism. The space of Hamiltonian symplectomorphisms
forms a group, denoted by $\text {Ham}(M, \omega)$.

In particular the path $\gamma_t$ above lies in $ \text {Ham}(M, \omega)$. It
is well-known 
that any smooth path $\{\gamma _{t} \}$ in $\text {Ham}(M, \omega)$ with
$\gamma_0=id$ arises in this way (is generated by $H_t: M \to \mathbb{R}$). Given such a path
$\{\gamma_t\}$, the \emph{Hofer length}, $L (\gamma _{t})$ is defined by
\begin{equation*}L (\gamma_{t}):= \int_0 ^{1} \max (H ^{\gamma} _{t}) -\min (H
^{\gamma}_t) dt,
\end{equation*}
where $H ^{\gamma} _{t}$ is a generating function for the path
$ \gamma_{0} ^{-1} \gamma _{t},$ $0\leq t \leq 1$.
The Hofer distance $\rho (\phi, \psi)$ is  defined by taking the
infimum of the Hofer length of paths from $\phi $ to $\psi$.
It is a deep theorem that the resulting metric is non-degenerate, (cf.
\cite{H, LM}). This gives $ \text {Ham}(M, \omega)$ the structure of a Finsler
manifold. We will be more concerned with a related measure of the path,
\begin{equation} \label {eq.pos.length}L ^{+} (\gamma_{t}):= \int_0 ^{1} \max (H
^{\gamma} _{t}), \end{equation} where $H ^{\gamma} _{t}$ is in addition normalized by the
condition 
\begin{equation*} \int _{M} H ^{\gamma _{t}}=0.
\end{equation*}

\subsection {Quantum Homology} \label{section.quantum.homology}
For a monotone symplectic manifold $(M, \omega)$ we set $QH (M) = H_* (M,
\mathbb{C})$. For us this is an ungraded vector space with a special product
called quantum product. For
integral generators $a,b \in H_* (M)$, this is the product defined by
\begin{equation} a*b = \sum _{A \in H_2 (M)} b_A e ^{-i \omega (A)},
\end{equation}
where $b _{A}$ is the homology class of the evaluation pseudocycle
from the pointed moduli space of
$J$-holomorphic $A$-curves intersecting generic pseudocycles representing
$a,b$,  for a generic $\omega$ tamed $J$. This sum is finite in the monotone
case: $\omega = k c_1 (TM)$, with $k>0$. 
The product is then extended to $QH (M)$ by linearity. For more technical details
see \cite{MS}.


\subsection {Quantum characteristic classes} Here, we give a brief overview of
the construction of the map $$\Psi: H_* (\Omega \text {Ham}(M, \omega), 
\mathbb{Q}) \to QH  (M),$$ originally defined in \cite{GS} and which
is a natural generalization of Seidel representation. Let $h:
B \to \Omega \text {Ham}(M, \omega)$ be a smooth cycle (the associated map $h: B \times S^1 \to \text {Ham}(M, \omega)$ is smooth), where $B$ is a
closed oriented smooth manifold, 
and let $p: P _{h} \to B$ be as in equation \eqref{eq.definitionPh}. 

Fix a
family $ \{j _{b,z}\}$ of almost complex structures on $$M \hookrightarrow P
_{h} \to B \times S ^{2},$$ fiberwise compatible with $\omega$. Given a smooth
family $\{ \mathcal {A} _{b} \}$, $ \mathcal {A} _{b}$ is a Hamiltonian connection on $X _{b}= p
^{-1} (p)$ we have an induced family of complex structures $\{J ^{ \mathcal {A}} _{b}\}$
defined as follows.
\begin{itemize} 
  \item The natural map
$\pi: (X _{b}, J_b ^{A})
\to (S ^{2}, j)$ is holomorphic for each $b$.  
\item  $J ^{ \mathcal {A}} _{b}$ preserves the horizontal subbundle $Hor _{b}
^{ \mathcal {A}}$ of $TX _{b}$ induced by $ \mathcal {A}$.
\item $J ^{ \mathcal {A}} _{b}$ preserves the vertical tangent bundle of $M
\hookrightarrow P _{h} \to B \times S ^{2}$ and restricts to the family $ \{j _{b,z}\}$.
\end{itemize}
\begin{definition} A family $ \{J _{b}\}$ is called \emph {
\textbf{$\pi$-compatible}} if it is $ \{J ^{ \mathcal {A}} _{b}\}$ for some
connection $ \mathcal {A}$ as above.
\end{definition}
The importance of this condition is that it forces bubbling to happen in the
fibers of $M \hookrightarrow X_b \to S ^{2}$, where it is controlled by
monotonicity of $M, \omega$.
 \begin{remark} \label{remark.fano} For the most part we work with a
 fixed symplectic manifold $ (M, \omega)$, which we will assume to be
 monotone. Also, for the purpose of the following definition the family $ \{j
 _{b,z}\}$ is fixed. However, it will be helpful in Section
 \ref{section.energy.flow}  to vary the families $ \{j _{b,z}\}$, $ \{\omega
 _{b,z}\}$ on  $M \hookrightarrow P _{h} \to B \times S ^{2}$, so long as each
 fiber $(M _{b, z}, \omega _{b,z}, j _{b,z})$ is Fano, i.e. $c_1 (TM)$ is
 positive on $j _{b,z}$ holomorphic curves. This will be done not for any 
 compactness or regularity reasons, but for  other geometric reasons. This will
 vary the notion of a $\pi$-compatible family $ \{J _{b}\}$, but not the map $\Psi$ below.
\end{remark}
The map $\Psi$ we now define measures part of the degree of quantum self
intersection of a natural submanifold $B \times M \subset P _{f}$. The entire
quantum self intersection is captured by the total quantum class of $P _{f}$,
discussed in \cite{GS}. We define $\Psi$ as follows: 
\begin{equation} \label {eq.def.ch.class} \Psi ([B, f])= \sum_ { {A}
\in j_* (H_2 ^{sect} (X))} b _{{ {A}}}  \cdot e ^{- i\mathcal {{C}}(
{A})}. 
\end{equation} Here,
\begin{itemize}
    \item $H ^{sect}_2 (X)$ denotes the section homology classes of $X$.
    \item $ \mathcal {C}$ is the coupling class of Hamiltonian fibration $M
    \hookrightarrow P _{f} \to B \times S ^{2}$, see
    \cite [Section 3]{kedra-2005-9}. Its restriction  to the fibers $X \subset P
    _{f}$ is uniquely determined by the condition \begin{equation} \label {eq.calC} i^* ( \mathcal {C})= [\omega],
\hspace {15pt} \int_M { \mathcal {C}}^ {n+1}=0 \in H^2 (S ^{2}). 
\end{equation} 
where $i: M \to X$ is the inclusion of fiber map, and the integral above denotes
the integration along the fiber map for the fibration $\pi: X \to S^2$.     
     \item  The map $j_*: H ^{sect}_2(X) \to H_2(P_f)$ is
 induced by inclusion of fiber. 
  \item The coefficient $b _{ {A}} \in H_* (M)$ is defined by duality:
 \begin{equation*} b _{ \widetilde{A}} \cdot _{M} c = ev_0 \cdot _{B \times M} [B]
 \otimes c, 
 \end{equation*} where 
\begin{align*} ev_0: \mathcal {M} _{0} (P _{h}, \widetilde{A}, \{J _{b}\}) \to B
\times M \\ ev_0 (u, b) = (u (0), b) 
\end{align*}
denotes the evaluation map from the  space
\begin{equation}  \label {eq.moduli.space} \mathcal {M} (P _{f}, {A},
\{J _{b}\})
\end{equation}
  of pairs $ (u, b)$, $u$ is a $J _{b}$-holomorphic section of $X
_{b}$ in class $ {A}$ and $\cdot _{M}, \cdot _{B \times M}$ denote the
intersection pairings. 
\item The family $ \{J _{b}\}$ is $\pi$-compatible in the sense
above.
\end {itemize}
\section {Hofer geometry and $\Psi$} \label {section.hoferandpsi}
We now show how to construct a $\pi$-compatible family $\{J ^{h}_{b}\}$ on $P _{h}$
naturally adapted to Hofer geometry of the map $h: B \to \ls$. This family is
induced from a family of Hamiltonian connections $ \{ \mathcal {A}_b, X
_{b}\}$, which are in turn induced by a family of certain closed forms $
\{ \widetilde{\Omega} ^{h} _{b}\}$, which we now describe.
 
The construction of this family mirrors the construction in
 Section 3.2 of \cite{GS}.  First we define a
family of forms $\{{ \widetilde{\Omega}} ^{\infty} _{
b} \}$ on $B \times M \times   D ^{2} _{\infty}$.
\begin{equation} \label {eq.coupling.family} { \widetilde{\Omega}}_{b} ^{h}|
_{D ^{2} _{\infty}} (x,r, \theta) = \omega - d ( \eta (r) H^b_ \theta (x))
\wedge d\theta \end{equation} Here, $H ^{b}_\theta$ is the generating Hamiltonian for $h(b)$, normalized so that $$\int
_{M}
H ^{b}_\theta \omega ^{n}=0, $$ for all $\theta$ and the function $\eta:
[0,1] \to [0,1]$ is a smooth function satisfying $$0 \leq \eta' (r),$$
and $$ \eta (r) = \begin{cases} 1 & \text{if } 1 -\delta \leq r \leq 1 ,\\
r ^{2}  & \text{if } r \leq 1-2\delta,
\end{cases}$$  for a small $\delta >0$. 

Note that under the
gluing relation $\sim$, $ (x, 1,\theta)_0 \mapsto (h (b, \theta)x,1,
\theta)_\infty$.
Thus,  $ \frac{\partial}{\partial \theta} \mapsto X _{H_\theta
^{b}} + \frac{\partial}{\partial \theta}$, $ \frac{\partial}{\partial x} \mapsto
(\gamma_\theta)_* (
\frac{\partial}{\partial x})$, and moreover $ \frac{\partial}{\partial r} \mapsto -
\frac{\partial}{\partial r}$.
We  leave it to the reader to check that the gluing relation $\sim$ pulls
back the form $ \widetilde{\Omega}_{b}^{h}| _{D ^{2} _{\infty}}$ to the form 
$\omega$ on the boundary $M \times \partial D^2_0$,
which we may then extend to $\omega$ on
the whole of $M
\times
D^2 _{0}$. Let $\{ \widetilde{\Omega} ^{h}_{b}\}$ denote the resulting family
on $X _{b}$. The forms $ \widetilde{\Omega} ^{h} _{b}$ on $X _{b}$ restrict
to $\omega$ on the fibers $M$ and the 2-form $\int _{M} (\widetilde{\Omega} ^{h}
_{b}) ^{n+1}$ vanishes on $S ^{2}$. Such forms are called \emph{coupling forms},
which is a notion due to Guillemin Lerman and Sternberg \cite{GuileminSternberg}. 
The form $ \widetilde{\Omega} ^{h} _{b}$ induces a connection on
$X _{b}$, by declaring horizontal subspaces to be those which are $
\widetilde{\Omega} ^{h} _{b}$-orthogonal to the vertical tangent spaces of
$\pi: X _{b} \to S ^{2}$. 
\begin{remark} \label{remark.couplingform} The induced connection is Hamiltonian
and moreover every Hamiltonian connection on $X _{b}$ is induced by a unique
coupling form in above sense, see \cite[Theorem 6.21]{MS2}. 
\end{remark}
We denote by $\{J
^{h} _{b} \}$ the  induced family of complex structures. An important property
of the family  $\{J ^{h} _{b}\}$ is that it is almost compatible with the
family $\{\Omega ^{h} _{b}, X _{b}\}$ defined by \begin{align} \label {eq.family} \Omega ^{h} _{b}| _{D ^{2} _{\infty}} = \widetilde{\Omega} ^{h}  _{b}| _{ D ^{2} _{\infty}} + (\max_x {H_\theta^b (x}))d \eta  \wedge d\theta, \quad
\Omega ^{h} _{b}| _{D ^{2} _{0}} = \widetilde{\Omega}
^{h}  _{b}| _{ D ^{2} _{0}}.
\end{align}
Where, \emph{almost compatible} means that $ \Omega ^{h} _{b} (v, J ^{h} _{b})
\geq 0$, $v \in TX _{b}$ and this inequality is strict for $v \in T ^{vert} X _{b}$. 

By the characterization of the class $ \mathcal {C}$ in \eqref{eq.calC}:
\begin{equation} \label {eq.coupling1} [ \widetilde{\Omega} ^{h} _{b}] = j ^{*}
(\mathcal {C}).
\end{equation}
Thus, 
\begin{equation} \label {eq.coupling2} [\Omega ^{h} _{b}] = j ^{*} (
\mathcal {C})  + [\pi ^{*}(\alpha_b)],
\end{equation}
where $j: X _{b} \to P _{h}$ is the inclusion map, and $\alpha _{b}$ is an
area form on $S ^{2}$ with 
\begin{equation} \label {eq.coupling3} L ^{+} (h (b)) =
\int _{S ^{2}} \alpha _{b}.
\end{equation}
\begin{lemma} \label{corollary.lower.bound} Let $ \{\Omega ^{h}_{b}\}$ and $
\{J ^{h} _{b}\}$ be as above, then we have the property that a vertical  $J
^{h} _{b}$-holomorphic section $u$ in the fiber $X _{b} \subset P _{h}$ gives a
lower bound
\begin {equation} \label {eq.lower.bound} -\mathcal {C}
([u]) \leq L ^{+}  (h (b).
\end {equation}           
\end{lemma} 
\begin{proof} Let $u: S^2 \to X _{b} \subset P _{h}$ be a holomorphic section,
then \begin{equation} \label {eq.positivity} 0 \leq \Omega ^{h}_{b} ([u]). 
\end{equation}
This follows from the almost compatible condition on $\Omega ^{h}_{b}$ and $J
^{h} _{b}$, from the fact that $u$ is a holomorphic map and by the fact that $J ^{h}
_{b}$ is $\pi$-compatible. Combining \eqref{eq.positivity} with
\eqref{eq.coupling2}, \eqref{eq.coupling3} we get:
$$0 \leq \int _{u} \Omega ^{h}_{b} = \mathcal {C}
([u]) + L ^{+} (h (b)).$$ 
\end{proof}
\begin{definition} \label{definition.energy} As the quantity $- \mathcal {C}
([u])$ is so important for us, we give it a name: $\energy$ of
$u$, or the \emph {\textbf{coupling energy}} of $u$.
\end{definition}
 Note that there are no fiber holomorphic curves in $P _{h}$ in class $A$ with
$\energy (A)> H _{\max}$, the maximum of the generating function $H$ of
$\gamma$.  This follows from the assumption that the pullback of the functional
$L ^{+}$ to  $B _{\gamma}$ attains its unique maximum at $\max$, $L ^{+} (h
(\max))=L ^{+} (\gamma)=H _{\max}$ and from the energy inequality
\eqref{eq.lower.bound}. On the other hand there is a special class $A _{\max}
\in H _{2} (X _{\gamma}) \subset H_{2} (P _{h})$ for which $\energy (A _{\max})=H
_{\max}$. We now describe this.

Let $x _{\max}$ denote the unique max of $H$, (cf. the discussion
preceding Theorem \ref{theorem.main}). There is a corresponding $\Omega ^{h}
_{\max}$-horizontal and thus holomorphic section $\sigma _{x_{\max}}$ of $\pi: X _{\max} \to S ^{2}$,
\begin{equation}  
\sigma _{x _{\max}} (z)= (\{x _{\max}\}, z
) _{0, \infty} \subset M \times D ^{2}_{0, \infty} \subset X _{\max} \quad \text
{for } z \in D ^{2} _{0, \infty}. 
\end{equation}

By \eqref{eq.coupling.family},  \eqref{eq.coupling1} $\energy ([\sigma_{x
_{\max}}])=H _{\max}$. Moreover, there are no other holomorphic sections $u$ in
$X _{\max}$ with $\energy ( [u])=H _{\max}$. This
observation is due to Seidel. For suppose otherwise, then by the proof of Lemma
\ref{corollary.lower.bound} we must have that $ \int _{u} \Omega ^{h}
_{\max}=0$, and so
\begin{multline} \label {eq.vanishing} 0=  \int _{u}
\omega - \eta (r) d H \wedge d \theta  - \int _{u} H  d  \eta  \wedge d\theta +
(\sup_x {H} (x))d \eta \wedge d\theta.
\end{multline}
Note that $u$ is necessarily horizontal, for otherwise
$\int _{u}\Omega ^{h} _{\max} >0$ (by the almost compatible property of $\Omega
^{h} _{\max}$ and $J ^{h} _{\max}$ and $\pi$-compatible property of $J ^{h}
_{\max}$). Hence the form $\omega - \eta (r)
d {H} \wedge d \theta$ must vanish on $u$, as the horizontal subspaces are
spanned by vectors $ \frac{\partial}{\partial r}, \frac{\partial}{\partial
\theta} + \eta (r)X _{H}$. Thus,
\eqref{eq.vanishing} can only happen if $u$ is $\sigma _{x _{\max}}$. In particular, the space $
\mathcal {M}
_{0} (X _{\max}, J ^{h} _{\max}, A_{\max})$ of 
unmarked holomorphic sections in $X _{\max}$ in homology class $A _{\max}=
[\sigma _{\max}]$ is identified with the point $x _{\max}$. We will also denote
by $A _{\max}$ the class $j_* (A _{\max}) \in H _{2} (P _{h})$, where $j: X \to P _{h}$ is the inclusion of fiber map.
 \begin{proposition} \label {lemma.max.energy} Let $h: B
_{\gamma} \to \ls$ be a smooth oriented cycle and $ \{\Omega ^{h} _{b}\}$ as above.
Suppose that the pullback by $h$ of the function $L ^{+}$ to  $B
_{\gamma}$  attains its
maximum at the unique point $\max \in B$, such that $h (\max)=\gamma$ then $ \mathcal {M} _{0} (P
_{h}, A _{\max}, \{J ^{h} _{b}\})$ lies over
$\max$ and is identified with $x _{\max}$.
\end{proposition} 
\begin{proof} 

The energy inequality \eqref{eq.lower.bound} shows that any  holomorphic curve 
in $P _{h}$ with $\energy$ $H _{\max}$ must lie in the fiber $X _{\max}$.
Consequently, by the discussion preceding the proposition, $$ \mathcal {M} 
_{0} (P _{h}, \{J ^{h} _{b}\}, A _{\max}) \simeq  \mathcal {M} _{0} (X
_{\max}, J ^{h} _{\max}, A _{\max}) \simeq x _{\max}.$$
\end{proof}  
\subsection {Proof of Theorem \ref{theorem.main}}
One example of Theorem \ref{theorem.main} that the reader may keep in mind comes from the
Hamiltonian $S ^{3}$ action on $S ^{2}$. Take $\gamma$ to be a one parameter
subgroup of $S ^{3}$, which is a great geodesic going around $S ^{3}$ once. The
subgroup $\gamma$ then acts on $S ^{2}$ by rotating it twice. The induced loop
$f \circ \gamma \subset \Omega \text {Ham}(S ^{2})$ is a critical point of $L
^{+}$, and this loop has a two parameter family of shortenings. More
specifically, the  unstable manifold $B _{\gamma}$ of $\gamma$ in $\Omega S ^{3}$ (for the
Riemannian energy functional) is two dimensional (by the Riemannian index theorem), and the pullback of the positive Hofer length
functional to  $B _{\gamma}$ is Morse at its maximum $\max \subset B
_{\gamma}$.
Why all this is the case
will be explained in the next  section. 
\begin{proof} 
It will be helpful to work with a special subset of $\pi$-compatible families $
\{J _{b}\}$. Let $ \widetilde{{C}}_b$ denote the space of coupling forms on $X
_{b}$, (see Remark \ref{remark.couplingform}) which restrict to $\omega$ over $M \times D ^{2} _{0}$ in the notation of \eqref{eq.definitionPh}. 
 A general element $ \widetilde{\Omega}
\in \widetilde{ {C}} _{b}$ is determined by a pair of families of
functions $G ^{\eta, \theta}, F ^{\eta, \theta}: M \to \mathbb{R}$, $$
\widetilde{\Omega}| _{D ^{2} _{\infty}}= \widetilde{\Omega} _{b} ^{h}| _{D ^{2} _{\infty}}  + d (G ^{\eta,
\theta} d\theta) + d(F ^{\eta, \theta} d \eta),$$  
where $ \widetilde{\Omega} _{b} ^{h}| _{D ^{2} _{\infty}}$ is defined in 
\eqref{eq.coupling.family} and $\eta$ is defined in the discussion following
\eqref{eq.coupling.family}.
 
  We set $ {C} _{b}
\subset \widetilde{ {C}} _{b}$ to be the subspace of those forms for which $$G ^{\eta,
\theta}= \eta \cdot G ^{\theta}, \text { for some $G _{\theta}$ and } 
\frac{d}{d\theta} F ^{\eta, \theta}= 0.$$ We also set $ {C}$ to be the space
of families $ \{ \widetilde{\Omega} _{b}\}$ on $P _{h}$, with each $ \widetilde{\Omega} _{b} \in \mathcal {C} _{b}$.
We have a function 
\begin{equation} \area: {C} _{b} \to \mathbb{R}, 
\end {equation}  
\begin{equation}
 \area (\widetilde{\Omega})= \inf \{\int _{S ^{2}} \alpha \, |
 \widetilde{\Omega} + \pi ^{*} (\alpha) \text { is symplectic} \}.
\end{equation}  
Let $\{\widetilde{\Omega}' _{b}\}$ be as in Lemma
\ref{lemma.perturbation} and sufficiently $C ^{\infty}$-close to $\{
\widetilde{\Omega} ^{h} _{b}\}$, then $ \{ \widetilde{\Omega} ' _{b}\}$ has
the property that the function 
\begin{equation} \label {eq.area} b \mapsto \area ( \widetilde{\Omega}'_{b}) 
\end{equation} 
on $B _{\gamma}$ attains its maximum  at the unique point $\max \in B
_{\gamma}$ and this function is Morse at $\max$. This readily follows from the
assumption that the pullback of $L ^{+}$ to $B _{\gamma}$ is Morse  at $\max$,
and by Lemma \ref{lemma.prop.ofarea} below.  Let $ \{J ' _{b}\}$ be the family
induced by $ \{ \widetilde{\Omega}' _{b}\}$, then by the proof of Proposition
 \ref{lemma.max.energy} $\mathcal {M} _{0} (P _{h'}, A _{\max},
\{J' _{b}\}) \simeq F _{\max}=\max$ and this moduli space is regular by
construction.  We thus verified the leading term of $\Psi (h)$ up to sign, which
depends  on the orientation of the cycle $B _{\gamma}$. The corrections are in
lower $\energy$ classes $A$, and consequently give rise to higher dimensional
moduli spaces via the dimension formula
\begin{equation} 2n+ \dim B_{\gamma} + <2 c_1 (T ^{vert} X), A>,
\end{equation}
and therefore are linearly independent of the leading term, (if they
contribute).
\end{proof}
\begin{lemma} \label{lemma.prop.ofarea} The coupling form $ \widetilde{\Omega}
^{h} _{\max}$ is a smooth point of the $\area$ functional on
${C}_{\max}$ and is critical.
\end{lemma}
Variation  $ \widetilde{\Omega} ^{s}$ in $ {C}_{\max}$ induces a
variation of the boundary monodromy maps $\gamma _{s}: [0,1] \to \text {Ham}(M,
\omega)$ induced by $\widetilde{\Omega} ^{s}| _{D ^{2} _{\infty}}$, which are
necessarily loops in $ \text {Ham}(M, \omega)$, since $
\widetilde{\Omega} ^{s}|  _{D ^{2}_0}= \omega$. By the properties of coupling
forms in $  {C} _{\max}$, the statement of the lemma is
equivalent to $\gamma$ being a smooth critical point for the $L ^{+}$ functional
on $\ls$. We leave the details to the reader. \begin{lemma} \label{lemma.perturbation} Let $h: B _{\gamma} \to \ls$
be as in Theorem \ref{theorem.main}. Then there is a family $ \{
\widetilde{\Omega}' _{b}\} \in  {C}$ on $P _{h}$ arbitrarily $C
^{\infty}$-close to $ \{ \widetilde{\Omega} ^{h} _{b}\}$, with $
\widetilde{\Omega}' _{\max} = \widetilde{\Omega} ^{h} _{\max}$, such that the
induced family $\{J ' _{b}\}$ is regular for $A _{\max}$-class curves. 
\end{lemma}
\begin{proof} 
Denote by ${\mathcal {B}}$ the space of pairs $(u, b)$, $u \in \mathcal {B}
_{b}$, with $\mathcal {B}
_{b}$ denoting the space of holomorphic sections of $X _{b}$.
We have a bundle
\begin{equation*} {\mathcal {E}} \to {\mathcal {B}},
\end{equation*}
whose fiber over $(u,b)$ is $\Omega ^{0,1} (S ^{2}, (u,b) ^{*} TX _{b})$, and
the section we call $ \mathcal {F} _{h}$, 
\begin{equation*} \mathcal {F} _{h} (u _{b})= \delbar _{J ^h_{b}} (u).
\end{equation*}

By the assumption that $h$ is Morse
at $\gamma$, and Proposition \ref{lemma.max.energy}, $$\mathcal {M} _{0} (P
_{h}, A _{\max}, \{J ^{h} _{b}\})$$ is a zero dimensional manifold  consisting
of a single point $u _{\max}$, which corresponds to the section $\sigma _{x
_{\max}}$ of $X _{\max} \subset P _{h}$. By assumption we have that $I
^{virt}(\gamma)=\dim B _{\gamma}$, where $I ^{virt} (\gamma)$ is the cokernel of
the vertical differential $$D \mathcal {F} _{h}| _{T _{u_{\max}} \mathcal {B}
_{\max}}.$$ And so zero is the expected dimension of $\mathcal {M} _{0} (P
_{h}, A _{\max}, \{J ^{h} _{b}\})$. Therefore, if we can perturb (abstractly) the section $ \mathcal {F} _{h}$, fixing it over $ \mathcal {B} _{\max} \subset { \mathcal {B}}$,  so that the corresponding vertical
differential at $u _{\max}$ has no kernel, then the perturbed section would be
necessarily transverse to the 0-section at $u _{\max}$. (The corresponding
differential would necessarily be Fredholm of index zero.)

More
specifically, we need a smooth  vertical (tangent to the fibers of $
{ \mathcal {E}} \to { \mathcal {B}}$) vector field
$\mathcal {V}$ along $ \mathcal {F} _{h}$ having the properties that it 
 vanishes over $ \mathcal {B} _{\max}$,  
and so that $
\mathcal {F} _{h}$ exponentiated along $\mathcal {V}$,  $$exp ^{ \mathcal {V}} _{t} ( \mathcal {F}
_{h}): { \mathcal {B}} \to { \mathcal {E}}$$ is transverse to
the $0$-section for all  sufficiently small time $t$. This is just a matter of
differential topology. The assumption that $ \mathcal {V}$ vanishes over $
\mathcal {B} _{\max}$  can be accommodated due to the fact that the vertical
differential $D \mathcal {F} _{h}$ restricted to $T _{u _{\max}} \mathcal {B} _{\max}$ has no
kernel. Which in turn follows from the fact that the vertical normal bundle to $u
_{\max}$ in $X _{\max}$ is holomorphic and all its Chern numbers are negative.
(Indeed this is another instance where the Morse condition on $H$ is crucial.)

We may also assume that $
\mathcal {V}$ vanishes  outside a neighborhood ${
\mathcal {B}} _{U _{\max}}$ of the curve $u _{\max}$ in $ \mathcal {B}$, with
all maps in  ${ \mathcal {B}} _{U _{\max}}$ lying over $U _{\max}$ a
contractible neighborhood of $\max$ in $B _{\max}$. Trivializing $P _{h}$ over
$U _{\max}$ we have $ \mathcal {B} _{U _{\max}}= \mathcal {B} _{\max} \times U
_{\max}$, where $ \mathcal {B _{\max}}= C ^{\infty} (S ^2, X _{\max})$.

We now show that the
perturbation $\mathcal {V}$ can be realized by perturbing the  family $
\{\Omega ^{h}_b\}$ and hence the induced family $ \{J ^{h} _{b}\}$.
Let 
\begin{equation*} { \mathcal {E}} \to (\mathcal {B} _{\max} \times U
_{\max}) \times  {C} _{\max} \equiv \widetilde{  {C}}
\end{equation*}
be the fibration whose fiber over $ (u, b, \widetilde{\Omega})$ is
$\Omega ^{0,1} _{J} (S ^{2}, u ^{*} TX _{\max})$ the space of $j, J$-anti-linear one forms, where $J$ is induced by $ \widetilde{\Omega}$. Let $$\mathcal
{F}:  \widetilde{  {C}} \to { \mathcal {E}}$$  be the map
\begin{equation*} \mathcal {F} (u, b, \widetilde{\Omega}) = \delbar _{J}
(u)
\end{equation*}

Denote by $T ^{vert} \widetilde{ \mathcal {C}}$
the vertical tangent bundle of $pr: \widetilde{  {C}} \to \mathcal {B}
_{\max} \times U _{\max}$. The vertical differential
\begin{equation*} D \mathcal {F}: T ^{vert} \widetilde{  {C}} \to T
^{vert} \mathcal {E},
\end{equation*} is a family of maps 
\begin{equation} \label {eq.onto} D \mathcal {F}  (u, b, \widetilde{\Omega}): T
_{ \widetilde{\Omega}}  {C} _{\max} \to \Omega _{J} ^{0,1} (S ^{2},
u ^{*} TX _{\max}). \end{equation}  By the proof of
Theorem 8.3.1 and Remark 3.2.3 in \cite {MS}, \eqref{eq.onto} is
onto for every $u$, $ \widetilde{\Omega}$. (One must of course work with the
appropriate Sobolev completions for this.)

 Let $ \mathcal {S}: U _{\max} \to \widetilde{  {C}}$ be the map $
 \mathcal {S} (b) = (u _{\max}, b, \widetilde{ \Omega} ^{h} _{b})$, induced by
 the family $\{ \widetilde{\Omega} ^{h} _{b} \}$ on $P _{h}$ over $ U _{\max}$.
And let $$ \mathcal {A} = D \mathcal {F} ^{-1} ( \mathcal {V}) \subset T ^{vert}
\widetilde{  {C}}$$ Since \eqref{eq.onto} is onto for every $u$,
$ \widetilde{\Omega}$, $ \mathcal {A}| _{ \mathcal {S}}$ fibers over $ \mathcal
{S}$. Let $ \mathcal {W}$ be any section, which we may think of as an infinitesimal
perturbation of the family $\{ \widetilde{\Omega} ^{h} _{b}\}$ for $b \in U
_{\max}$. This perturbation extends by vanishing perturbation outside $U
_{\max}$. The infinitesimal perturbation $ \mathcal {W}$ is the one
we were looking for and so we are done.
\end{proof} 
\subsection {Proof of Corollary \ref{corollary.max.length}}
The first part is immediate. To prove the second part
note that $h: B _{\gamma} \to \Omega \text {Ham}(M, \omega)$ has max length
measure $H _{\max}$. On the other hand if the max length measure of the map $h$
could be reduced below $H _{\max}$ by moving it in its homology 
class to say $h':  B \to \Omega \text {Ham}(M, \omega)$, then this would destroy
the contribution to $\Psi (h)=\Psi (h')$ in the $\energy$ $H _{\max}$, because by Corollary \ref{corollary.lower.bound}
there would simply be no vertical $ \{J ^{h'} _{b}\}$-holomorphic curves in $P _{h'}$ with $\energy =H
_{\max}$, this is a contradiction.  
\qed


\subsection {Proof of Theorem \ref{theorem.index}}
We clearly have $I^{H}(\gamma) \geq \dim B_{\gamma}$. 
Suppose $I ^{H} (\gamma) > \dim B_{\gamma}=m$, then there a subspace $N \subset
T _{\gamma} \ls$ such that $N \supset h_* T _{\max} B _{\gamma}$, $\dim N= \dim
B _{\gamma}+1$ and such that the Hessian of $L ^{+}$ is negative definite on
$N$. We homotop the map $h$ to a map $h'$ so that $L ^{+}(h' (b)) < L ^{+}
(\gamma)$ for all $b \in B _{\gamma}$, which kills the contribution to $\Psi
(h)$ in the energy $H _{\max}$ by the proof of \ref{corollary.max.length}. This
will conclude the proof. 

Let $\phi: D ^{m} \to B _{\gamma}$ be a
chart containing $\max$. We may homotop $h$ to a map $\widetilde{h}$ with
the same image as $h$, with $ \widetilde{h}$ being the constant map to $\gamma$
on $\phi (D ^{m})$.
Let
$p: D ^{m} \to N -0$ be an embedding so that $$p: \partial D ^{m} \to h_* (S
^{m} \subset T _{\max} B _{\gamma})$$ is a degree one map, where the unit sphere $S ^{m}$ is determined by the trivialization $\phi$.
Under the identification given by $\phi$
extend $p$ to any  smooth map $$ \widetilde{p}: B _{\gamma} \to T \ls.$$ Now
move $ \widetilde{h}$ along $\widetilde{p}$ by exponentiating for a sufficiently small time,
then the exponentiated map $h'$ will have the required property.
\qed



%
%
%
%
%
 \section {Morse theory on $\Omega \text {Ham}(G/T)$ and $\Omega G$.} \label
 {section.energy.flow}
%
Let $M=G/T$, where $T$ is its maximal torus. There a symplectic
structure on $G/T$, inherited from that of $T ^{*} G$ by 
symplectic reduction of the natural $G$ action on $T ^{*}G$, ($G/T$ is the
generic leaf of the symplectic reduction.)   The leaves of the the symplectic
reduction of $T ^{*}G$ and hence $G/T$ can be identified with orbits of the 
coadjoint action of $G$ on $\frak {g} ^{*}$, where  $ \frak g$ is the Lie
algebra of $G$. The symplectic structure is then induced from  a natural 2 form
on $\frak {g} ^{*}$ called the Kirillov form, (see \cite{Arnold.methods}). 

Let $G$ be semisimple and $\mathcal {O} _{p _{0}}$ a coadjoint orbit of $p_0
\in \mathfrak g ^{*}$ by $G$. Then $G$ acts on $\mathcal {O} _{p_0}$ by $ \phi
_{g} (p) =Ad ^{*} _{g ^{-1}}(p)$. With the infinitesimal action $X _{\eta} (p)= -ad ^{*} _{\eta} (p)$ for $\eta
 \in \frak g$. The generating
 function is defined by $H _{\eta} (p)= p(\xi)$ and it is normalized, as the
 map $$\eta \mapsto \int _{O _{\xi}} H _{\eta} (p) = \int _{O _{\xi}} p (\eta)$$ defines an element of $ \frak g ^{*}$,
which is clearly invariant under the coadjoint action of $G$ and so            
must be 0 (since $\frak g$ has no center).\footnote {I would like to thank Yael
Karshon for suggesting this argument.}  

  Suppose now  $\xi= \frac{d}{dt}
 \arrowvert_0 \gamma \in \mathfrak {g}$, where $\gamma$ is a one parameter
 subgroup and let $ \mathcal {O _{\xi}}$ denote the coadjoint orbit of the
 covector $< \frac{\xi}{||\xi||}, \cdot>$. In this case the maximum of the generating function $H _{\xi}$ on $ \mathcal {O _{\xi}}$ is
 $< \frac{\xi}{||\xi||} , \xi> = ||\xi||$. Moreover, we get an 
 inequality relating positive Hofer norm with Riemmanian norm,
 \begin
 {equation} \label {eq.norm} ||\eta|| ^{+} \leq ||\eta||,
\end {equation}
for any $\eta$, 
where $$||\eta|| ^{+} = \max _{ \mathcal {O} _{\xi}} H _{\eta}.$$  
In this discussion the symplectic manifold $ \mathcal {O _{\xi}}$ depends on
 $\xi$. If we make an additional  assumption, that the subgroup of $G$ fixing
$<\xi, \cdot>$, under the coadjoint action is $T$, then  we can identify $O
_{\xi} \simeq (G/T, \omega _{\xi})$. Moreover, this condition is generic
in $ \frak {g} ^{*}$, from which it follows that the symplectic forms $\omega
_{\xi}$ are deformation equivalent. Also $(G/T, \omega _{\xi}, j _{\xi})$,
are Fano for an integrable complex structure $j _{\xi}$ depending smoothly on
$\omega _{\xi}$. Therefore, we may regard $ \mathcal {O} _{\xi}$ as simply
$G/T$ for our  purposes of quantum homology and the map $\Psi$, (cf. Remark
\ref{remark.fano}) 
\subsection {Morse theory on $\Omega G$}
Let $h: B _{\gamma} \to \Omega G$ be the pseudocycle corresponding to the
unstable manifold of $\gamma$ for the energy flow on $\Omega G$. (See
the discussion in Section \ref{section.special.case}. It is necessarily a
pseudocycle, since all the indexes of critical points of $E$ are even.) As
before we denote by $\max \in B _{\gamma}$ the point $h (\max)=\gamma$.
\begin{theorem} \label{lemma.normalized} 
Let $G$ be a semi simple compact Lie
group, then the positive Hofer length functional $L ^{+}: B _{\gamma} \to \mathbb{R}$ 
(its pullback from $\Omega \text {Ham}(O _{\xi})$) is Morse at $\max$.
Moreover, if the centralizer of $\gamma$ is the torus then the indexes
$I ^{virt} (f(\gamma))$  (cf. eq. \eqref{eq.virtualindex}) and the Riemannian index of
$\gamma$ coincide.  In other words: \begin {equation} I ^{virt}(f ( \gamma))=\dim B _{\gamma}.
\end {equation}
\end{theorem}
\begin{proof}    
By  \eqref{eq.norm} we have $L ^{+} (\gamma _{b}) \leq
L (\gamma _{b})$ for $\gamma _{b}$ any loop in $B _{\gamma}$ (or $\Omega G$),
where $L$ is the Riemmanian length functional on $\Omega G$. Since $L ^{+}
(\gamma)= L (\gamma)$, the first part of the theorem will follow if the
restriction of $L$ to $B _{\gamma}$ is Morse at $\gamma$. This is intuitively
clear as the restriction of the energy functional $E$ to $B _{\gamma}$
is Morse at $\gamma$ since $E$ is a Morse-Bott function on $\Omega G$. Here are
the details. Let $\gamma _{t}$ be a smooth variation of $\gamma=\gamma_0$ in $B
_{\gamma}$. Applying Cauchy-Schwarz inequality, \begin{equation*} (\int_a ^{b} fg \,
d\theta) ^{2} \leq (\int _{a} ^{b} f ^{2} d\theta) (\int _{a} ^{b} g ^{2} \,
d\theta),
\end{equation*} 
with $f (\theta)=1$ and $g (\theta) = || \frac{d}{d\theta}| _{\theta} \gamma
_{t} (\theta)||$, we get \begin{equation*} L (\gamma _{t})^{2} \leq E (\gamma
_{t}),
\end{equation*}  since $\gamma$ is parametrized from
$0$ to $1$. Both sides are the same for $t=0$, (since $\gamma$ is a
geodesic and so parametrized by arclength) and the derivatives of both sides
are 0 at $t=0$ since $\gamma$ is  critical for both $L$ and $E$. It follows
that \begin{equation*} \frac{d ^{2}}{dt ^{2}}|_0 L (\gamma _{t}) ^{2} =
2 L (\gamma) \cdot \frac{d ^{2}}{dt ^{2}}|_0 L (\gamma _{t}) \leq \frac{d
^{2}}{dt ^{2}}|_0 E <0,
\end{equation*}
and so $$\frac{d ^{2}}{dt ^{2}} |_0 L (\gamma_t) <0.$$

 We now prove the second part of the theorem.
Let $\gamma$ be generic and $\xi$
the corresponding element in $ \frak g$. In order to compute $I ^{virt}
(f \circ \gamma)$ we need to understand the weights of the coadjoint action of
$\gamma$ on the tangent space $T _{p} O _{\xi}$, where $p$ is the maximal fixed point
$p=< \frac{\xi}{||\xi||}, \cdot>$. Since the maps $Ad _{g} ^{*}$ are linear, this
action can be identified with the action of $\gamma$ on a subspace of $T _{0}
\frak g ^{*} \equiv \mathfrak{g} ^{*}$. Moreover,
under the
identification of $ \frak g ^{*}$ with $ \frak g$ using the $Ad$-invariant
inner product $<,>$ on $\frak g$ the coadjoint action by $Ad _{g} ^{*}$ on
$\frak g ^{*}$ corresponds to the adjoint action by $Ad _{g ^{-1}}$ on $\frak
{g}$ and so the coadjoint action of $\gamma$ on $\frak g ^{*}$ corresponds to
the adjoint action of $\gamma ^{-1}$ on $\frak g$. More specifically, we want the adjoint
action on a certain subspace of $T _{p} \subset \frak g$ which corresponds
under all these identifications to $T _{p} O _{\xi}$. In fact this subspace can
be determined synthetically as follows. Write
\begin {equation} \mathfrak g= \frak t \oplus \bigoplus _{\alpha} \frak g
_{\alpha},   
\end {equation}  
where $\frak t$  is the maximal Abelian subalgebra of $\frak g$ containing
$\xi$ and $\frak g _{\alpha}$ is a subspace of $\frak g$ on which $\gamma
^{-1}$ is acting by $e ^{ \alpha 2 \pi i \theta}$. (So that $ \frak t$
corresponds to $\alpha=0$.) Now, $T _{p}$ is invariant under the
adjoint action of $\gamma ^{-1}$ and  all the weights $\alpha$ are
necessarily non zero on $T _{p}$ and are negative. The latter is due to the
fact that the function $H _{\xi}$ on $O _{\xi}$ is Morse at its maximum $p$, 
which together with our convention $ X _{H}=-J \grad H$ implies that
the weights are negative.
The subspace $T
_{p}$ must then simply be $$T _{p}= \bigoplus _{\alpha} \frak g  _{\alpha}.$$ 
The virtual index is then by definition
\begin {equation*} \sum _{\alpha} 2|\alpha| -2. 
\end {equation*}
Using the index theorem in Riemannian geometry one can show that
this is the Riemannian index of the geodesic $\gamma$ of $G$, see for example
proof of Bott periodicity,  \cite [Section 23]{Morse.theory}.
\end{proof}   
\begin{corollary}  \label {corollary.ultimate} Let $G$ be a semi simple Lie
group and $\gamma$ a generic $S ^{1}$ subgroup. Then the pseudocycles $f \circ
h: B _{\gamma} \to \Omega \text {Ham}(G/T)$ satisfy the hypothesis of Theorem
\ref{theorem.main}.
\end{corollary}  
\begin{remark} Strictly speaking the map $\Psi$ in Theorem 1.2 is only defined
here and in \cite{GS} on cycles $f: B \to \ls$, where $B$ is a closed smooth
manifold. However there is no essential difficulty in extending this to
appropriate pseudocycles. The details of this will be given in a more general
context in the upcoming paper \cite{fieldtheory}.
\end{remark}
\bibliographystyle{siam}  
\bibliography{link} 
\end{document}